\documentclass[12pt]{article}

\usepackage{fullpage}
\usepackage{amssymb}
\usepackage{amsmath}
\usepackage[usenames]{color}
\usepackage{graphicx}
\usepackage{amsthm}
\usepackage{amsfonts}
\usepackage[colorlinks=true,
linkcolor=webgreen,
filecolor=webbrown,
citecolor=webgreen]{hyperref}

\definecolor{webgreen}{rgb}{0,.5,0}
\definecolor{webbrown}{rgb}{.6,0,0}
\definecolor{red}{rgb}{1,0,0}

\usepackage{color}
\newtheorem{theorem}{Theorem}

\newtheorem{lemma}[theorem]{Lemma}
\newtheorem{proposition}[theorem]{Proposition}
\newtheorem{Observation}[theorem]{Observation}

\theoremstyle{definition}
\newtheorem{definition}[theorem]{Definition}

\theoremstyle{remark}
\newtheorem{remark}[theorem]{Remark}
\theoremstyle{conj}

\theoremstyle{he}

\newcommand{\ord}{\mbox{ord}}

\begin{document}

\title{\bf On prime factors of Mersenne numbers}
\author{Ady Cambraia Jr\footnote{Departamento de Matem\'{a}tica, Universidade Federal de Vi\c cosa, Vi\c cosa-MG 36570-900, Brazil}, Michael P. Knapp\footnote{Department of Mathematics and Statistics, Loyola University Maryland, 4501 North Charles Street, Baltimore MD 21210-2699, USA}, Ab\'{i}lio Lemos$^{*}$,\\
 B. K. Moriya$^{*}$ and Paulo H. A. Rodrigues\footnote{Instituto de Matemática e Estatística, Campus Samambaia, CP 131, CEP 74001-970, Goiânia, Brazil}\\
 ady.cambraia@ufv.br\\
 mpknapp@loyola.edu\\
 abiliolemos@ufv.br\\
 bhavinkumar@ufv.br\\
 paulo\_rodrigues@ufg.br}
\maketitle

	\begin{abstract}
Let $(M_n)_{n\geq0}$ be the Mersenne sequence defined by $M_n=2^n-1$. Let $\omega(n)$ be the number of distinct prime divisors of $n.$ In this short note, we present a description of the Mersenne numbers satisfying $\omega(M_n)\leq3$. Moreover, we prove that the inequality, given $\epsilon>0$, $\omega(M_n)> 2^{(1-\epsilon)\log\log n} -3
$ holds for almost all positive integers $n$. Besides, we present the integer solutions $(m,n,a)$ of the equation $M_m+M_n=2p^a$ with $m,n\geq2$, $p$ an odd prime number and $a$ a positive integer.
\end{abstract}


\maketitle

\section{Introduction}

Let $(M_n)_{n\geq0}$ be the {\it Mersenne sequence} defined by $M_n=2^n-1$, for $n\geq0$. A simple argument shows that if $M_n$ is a prime number, then $n$ is a prime number. When $M_n$ is a prime number, it is called a Mersenne prime. Throughout history, many researchers sought to find Mersenne primes. Some tools are very important for the search for Mersenne primes, mainly the Lucas-Lehmer test. There are papers (see for example \cite{Br,Eh,Wa}) that seek to describe the prime factors of $M_n$, where $M_n$ is a composite number and $n$ is a prime number.

Additionally, some papers seek to describe prime divisors of the Mersenne number $M_n$, where $n$ is a composite number (see for example \cite{FLS,MP,Po,Sc,St}). In this paper, we propose to investigate the function $\omega(n)$, which refers to the number of distinct prime divisors of $n$, applied to $M_n$. We prove that if $n\neq2,6$ and $p_1^{\alpha_1}\cdots p_s^{\alpha_s}|n$, where the $\alpha_i'$s are positive integers, $p_1,\dots, p_s$ are distinct prime numbers and $\sum_{i=1}^{s}\alpha_i=t$, then $\omega(M_{n})\geq t$ if  $s=1$, and 
$$\omega(M_{n})\geq t+\displaystyle\sum_{i=2}^{s}{s \choose i}, \mbox{ if }  s > 1.$$
This lower bound is sharp for some cases.  For example, $\omega(M_{4})=\omega(M_{9})=\omega(M_{49})=2$ and $\omega(M_{8})=\omega(M_{27})=\omega(M_{10})=\omega(M_{14})=\omega(M_{15})=\omega(M_{21})=3$.

The equations $F_n=y^q$ and $L_n=y^q,$ where $q$ is a prime number, $F_n$ is a Fibonacci number and $L_n$ is a Lucas number, have been studied by several authors.  One may see, for example, \cite {ch,lj1,lj2,lj3,lj4,lf,pe}. The complete solution for this problem was obtained by Bugeaud, Mignotte, and Siksek \cite{bms}), by combining the classical approach to exponential Diophantine equations (linear forms in logarithms, Thue equations, etc.) with a modular approach based on some of the ideas of the proof of Fermat's Last Theorem. When we consider $M_n=y^q,$ we see that this is equivalent to the equation $2^n-y^q=1$. The solution in this case is much simpler, since the Catalan Conjecture (proved by Mih{\v a}ilescu \cite{Mi} in 2002) tells us that the only solution of the equation $x^m-y^n=1$, with $m,n>1$ and $x,y>0$ is $x=3,m=2,y=2,n=3$. This motivates us to raise the next question in the flavour of these studies.
We believe that the next step is to study equations of the form $ M_m+M_n=2y^q.$ Here we consider the equation $ M_m+M_n=2p^a,$ where $p$ is an odd prime and $a\in\mathbb{N}$. To the best of our knowledge, this equation has not been studied anywhere in the literature. If $p\equiv 1 \pmod 4$, we prove that the only solutions of the underlined equation are of the form $(m,n,a) = (2,2^b+1,1),$ where $p=2^{2^b}+1$. If $p\equiv 3 \pmod 4$, there is more than one possible form of the solution. Among others, one possible case could be $ (m,n,a)=(k+1,n,a)$ if $p^a+1>2^k.$

\section{Preliminary results}

We investigate $\omega(M_{n}),$ the number of distinct prime divisors of $M_n.$ We start by stating some well-known facts and results. The first result is the well-known Theorem XXIII of \cite{Ca}, obtained by Carmichael.

\begin{theorem}\label{L1}
If $n\neq1,2,6$, then $M_n$ has a prime divisor which does not divide any $M_m$ for $0<m<n$. Such a prime is called a primitive divisor of $M_n$.  
\end{theorem}
We also need the following results:
\begin{equation} \label{L2}
d=\gcd(m,n) \Rightarrow \gcd(M_m,M_n)=M_d
\end{equation}
 
\begin{proposition}\label{P1}
If $1<m<n$, $\gcd(m,n)=1$ and $mn\neq6$, then $\omega(M_{mn})>\omega(M_m)+\omega(M_n)$.
\end{proposition}
\begin{proof}
As $\gcd(m,n)=1$, it follows that $\gcd(M_m,M_n)=1$ by \eqref{L2}. Now, according to Theorem \ref{L1},
we have a prime number $p$ such that $p$ divides $M_{mn}$ and $p$ does not divide $M_mM_n$. Therefore, the proof of proposition is completed.
\end{proof}

Mih{\v a}ilescu \cite{Mi} in his proof of the famous Catalan Conjecture proved the following.
\begin{theorem}\label{L3}
The only solution of the equation $x^m-y^n=1$, with $m,n>1$ and $x,y>0$ is $x=3,m=2,y=2,n=3$.
\end{theorem}

For $x=2$, Theorem \ref{L3} ensures that there is no $m>1$, such that $2^m-1=y^n$ with $n>1$. 

\begin{lemma}\label{L4}
Let $p,q$ be prime numbers. Then, 
\begin{enumerate}
\item [(i)] $M_p\nmid\left(M_{pq}/M_p\right)$, if $2^p-1\nmid q;$
\item [(ii)] $M_p\nmid\left(M_{p^3}/M_p\right)$.
\end{enumerate}
\end{lemma}
\begin{proof}
$(i)$ We note that $M_{pq}=(2^p-1)(\sum_{k=0}^{q-1}2^{kp})$. Thus, if $(2^p-1)|(\sum_{k=0}^{q-1}2^{kp})$, then 
$$
(2^p-1)\Big|\left(\sum_{k=0}^{q-1}2^{kp}+2^p-1\right)=2^{p+1}\left(2^{pq-2p-1}+\cdots+2^{p-1}+1\right).
$$
Since $2^{pq-2p-1}+\cdots+2^{p-1}+1 \equiv (q-2)2^{p-1}+1 \pmod{2^p -1}$, we have $(2^p-1)|\left((q-2)2^{p-1}+1\right)$.  Therefore,
$$
(2^p-1)|\left((q-2)2^{p-1}+1+(2^p-1)\right)=2^{p-1}q,
$$
i.e., $2^p-1|q$.  Therefore, the proof of $(i)$ is completed.\\
The proof of $(ii)$ is analogous to the proof of $(i)$.

\end{proof}
\begin{remark}\label{L5}
It is known that all divisors of $M_p$ have the form $q=2lp+1$, where $p$ is an odd prime number and $l\equiv 0\mbox{ or }-p\pmod4$.
\end{remark}
The first part of Remark \ref{L5} was first obtained by Euler (see Theorema 9 and Corollarium 5 in \cite{Eu}). Euler does not actually write the coefficient 2, but clearly as the expression must be odd we may do so without loss of generality. 

The modern proof of this result is as follows. Let $q_1$ be a prime divisor of $q.$ By Fermat's little theorem, $q_1$ is a factor of $2^{q_1-1}-1.$ Since $q_1$ is a factor of $2^p-1,$ and $p$ is prime, it follows that $p$ is a factor of $q_1-1$ so $q_1 \equiv 1 \pmod p.$ Furthermore, since $q_1$ is a factor of $2^p-1,$ which is odd, we know that $q_1$ is odd. Therefore, $q_1 \equiv 1 \pmod{2p}.$ But, since this result is true for all prime divisors of $q,$ i.e., for $q_1=2l_1p+1$ and $q_2=2l_2p+1$ prime divisors of $q$ not necessarily distinct, we have $q_1q_2=2l_{12}p+1,$ where $l_{12}=2l_1l_2p+l_1+l_2.$ Applying recursively for all prime divisors of $q,$ we get $q=2pl+1,$ for some $l\in\mathbb{N}.$   

Now, the second statement is a consequence of the fact that all prime divisors of $q$ are congruent to $\pm1\pmod8.$ Let us prove it here. We have $2^{p+1}\equiv 2 \pmod{q_1},$ where $q_1$ is a prime divisor of $q.$ Thus, $2^{\frac{1}{2}(p+1)}$ is a square root of 2 modulo $q_1.$ By the theory of quadratic residues, we have $q_1\equiv\pm1\pmod 8.$ Since this result is true for all prime divisors of $q,$ we have $q=2lp+1\equiv \pm1\pmod 8,$ for some $l\in\mathbb{N}.$ Now, if $2lp+1\equiv \pm1\pmod 8,$ then $lp\equiv 0\pmod 4$ or $lp\equiv 3\pmod 4.$ The first case implies $l\equiv 0\pmod4,$ since $p$ is odd. The second case implies $l\equiv -p\pmod4,$ since $-p^2\equiv 3\pmod4.$

\section{Mersenne numbers rarely have few prime factors.}

We present below a result which is a consequence of Manea's Theorem, which we shall state next.  We shall get the multiplicity $v_q(M_n)$ for a given odd prime $q$ and a positive integer $n.$

\begin{definition}
Let $n$ be a positive integer. The $q$-adic order of $n,$ denoted by $v_q(n),$ is defined to be the natural number $l$ such that $q^l\mid\mid n,$ i.e., $n=q^lm$ with $\gcd(q,m)=1.$ 
\end{definition}

\begin{lemma}[Theorem 1 \cite{man}]\label{man}
Let $a$ and $b$ be two distinct integers, $p$ be a prime number that does not divide $ab$, and $n$ be a positive integer. Then

\begin{enumerate}
\item if $p\neq 2$ and $p|a-b$, then
\[v_p(a^n-b^n)=v_p(n)+v_p(a-b);\]
\item if $n$ is odd, $a+b\neq 0$ and $p|a+b$, then \[v_p(a^n+b^n)=v_p(n)+v_p(a+b).\]
\end{enumerate}
\end{lemma}

\begin{definition}
For a given prime number $q$ and an integer $a\in \{1,\ldots, q-1\}$ the number $\ord_q(a)$ is defined to be the smallest positive integer $t$ such that $a^t\equiv 1\pmod q$.
\end{definition}

\begin{proposition}\label{t7}
Let $q \neq 2$ be a prime number.  Define $m = \ord_q(2)$ and $w = v_q(2^m - 1)$.  Let $n \in \mathbb{N}$, and write $n = q^l n_0$, with $\gcd(q,n_0) = 1$.  Then
\[  v_q(M_n) = v_q(2^n - 1) = 
\begin{cases}
0 & \mbox{if $m \nmid n$}\\
l + w & \mbox{if $m | n$.}
\end{cases}  \]
\end{proposition}

\begin{proof}
By elementary number theory, we know that $2^n \equiv 1 \pmod{q}$ if and only if $\ord_q(2) | n$.  This proves the first line of the formula.\\

Now, suppose that $m|n$ and write $n = mt$.  Then we have
\[  M_n  =  (2^m)^t - 1^t.  \]
\noindent By Theorem \ref{man} (with $a = 2^m$ and $b=1$), we have
\begin{eqnarray*}
v_q(M_n)  &  =  &  v_q(t) + v_q(2^m - 1)\\
  &  =  &  l + w.
\end{eqnarray*}
\noindent  This completes the proof.
\end{proof}

We will prove the following result for a lower bound of $\omega(M_n).$ 

\begin{theorem}\label{T14}
Let $\epsilon$ be a positive number. The inequality
$$
\omega(M_n)> 2^{(1-\epsilon)\log\log n} -3
$$
holds for almost all positive integer $n$. 
\end{theorem}

\begin{theorem}[Theorem 432, \cite{HW}] \label{L6}
Let $d(n)$ be the total number of divisors of $n$. If $\epsilon$ is a positive number, then 
$$
2^{(1-\epsilon)\log\log n}< d(n)< 2^{(1+\epsilon)\log\log n}
$$
for almost all positive integer $n$. 
\end{theorem}
  
\begin{proof}
(Proof of Theorem \ref{T14}).
According to Theorem \ref{L1}, we know that if $h|n$ and $h\neq1,2,6$, then $M_h$ has a prime primitive factor. This implies that
$$
\omega(M_n)\geq d(n)-3
$$ 
Consequently, by Theorem \ref{L6}, we have
$$
\omega(M_n)> 2^{(1-\epsilon)\log\log n} -3
$$
for almost all positive integer $n$.   
\end{proof}


\section{Mersenne numbers with $\omega(M_n)\leq3$}

In this section, we will characterize $n$ for $\omega(M_n)=1,2,3.$ 

\begin{theorem}\label{T1}
The only solutions of the equation
$$
\omega(M_n)=1 
$$
are given by $n$, where either $n=2$ or $n$ is an odd prime for which $M_n$ is a prime number of the form $2ln+1$, where $l\equiv 0\,\,\mbox{or}\,\,-n\,\,\pmod 4$.
\end{theorem}

\begin{proof}
The case $n=2$ is obvious. For $n$ odd, the equation implied is $M_n=q^m$, with $m\geq1$. However, according to Theorem \ref{L3}, $M_n\neq q^m$, with $m\geq2$. Thus, if there is a unique prime number $q$ that divides $M_n$, then $M_n=q$, and $q=2ln+1$, where $l\equiv 0\,\,\mbox{or}\,\,-n\,\,\pmod 4$, according to Remark \ref{L5}. 
\end{proof}

\begin{proposition}\label{P2}
Let $p_1,p_2, \dots, p_s$ be distinct prime numbers and $n$ a positive integer such that $n\neq2,6$. If $p_1^{\alpha_1}\cdots p_s^{\alpha_s}|n$, where the $\alpha_i'$s are positive integers and $\sum_{i=1}^{s}\alpha_i=t$, then
$$\omega(M_{n})\geq\left\{\begin{array}{lcc}
t, & \mbox{if} & s=1\\
t+\displaystyle\sum_{i=2}^{s}{s \choose i}, & \mbox{if} & s > 1\end{array}\right..$$
\end{proposition}

\begin{proof}
According to Theorem \ref{L1}, we have
$$\omega\left(M_{p_i^{\alpha_i}}\right)>w\left(M_{p_i^{\alpha_i-1}}\right)>\dots> \omega\left(M_{p_i}\right)\geq1,$$ 
for each $i\in\left\{1,\dots, s\right\}$. Therefore, $\omega(M_{p_i^{\alpha_i}})\geq\alpha_i$ and this proves the case $s = 1$. Now,  we observe that $\gcd\left(\prod_{i\in I}p_i,\prod_{j\in J}p_j\right)=1,$ for each pair $\emptyset\neq I,J\subset\{1,\dots,s\}$ with $I\cap J=\emptyset$. Then it follows from Theorem \ref{L1} and Proposition \ref{P1} that
$$\omega(M_{n})\geq\sum_{i=2}^{s}{s \choose i} + t, \,\mbox{if}\, s>1.$$ 
\end{proof}

\begin{theorem}\label{T2}
The only solutions of the equation
$$
\omega(M_n)=2 
$$
are given by $n=4,6$ or $n=p_1$ or $n=p_1^2$, for some odd prime number $p_1$. Furthermore, 
\begin{enumerate}
\item [(i)] if $n= p_1^2$, then $M_{n}=M_{p_1}q^t$, $t\in\mathbb{N}$. 
\item [(ii)] if $n=p_1$, then $M_{n}=p^sq^t$, where $p,q$ are distinct odd prime numbers and $s,t\in\mathbb{N}$ with $\gcd(s,t)=1$. Moreover, $p,q$ satisfy $p=2l_1p_1+1, q=2l_2p_1+1$, where $l_1,l_2$ are distinct positive integers and $l_i\equiv 0\,\,\mbox{or}\,\,-p_1\,\,\pmod 4$.  
\end{enumerate}
  
\end{theorem}
\begin{proof}
This first part is an immediate consequence of Proposition \ref{P2}.

$(i)$ If $\omega(M_n)=2$, with $n=p_1^2$, then on one hand $M_n=p^sq^t$, with $t,s\in\mathbb{N}$. On the other hand, by Theorem \ref{L1} $\omega(M_{p_1^2})>\omega(M_{p_1})\geq1$, i.e., $M_{p_1}=p$, by Theorem \ref{L3}. Thus, according to Lemma \ref{L4}, $M_n=M_{p_1}q^t=pq^t$, with $t\in\mathbb{N}$. 

$(ii)$ If $\omega(M_n)=2$, with $n=p_1$, then $M_n=p^sq^t$, with $t,s\in\mathbb{N}$. However, according to Theorem \ref{L3}, we have $\gcd(s,t)=1$. The remainder of the conclusion is a direct consequence of Remark \ref{L5}.
 
\end{proof}

\begin{theorem}\label{T3}
The only solutions of the equation
$$
\omega(M_n)=3 
$$
are given by $n=8$ or $n=p_1$ or $n=2p_1$ or $n=p_1p_2$ or $n=p_1^2$ or $n=p_1^3$, for some distinct odd prime numbers $p_1<p_2$. Furthermore,  
\begin{enumerate}
\item [(i)] if $n=2p_1$ and $p_1\neq3$, then $M_{n}=3M_{p_1}k^r=3qk^r$, $r\in\mathbb{N}$ and $q,k$ are prime numbers. 
\item [(ii)] if $n=p_1p_2$, then $M_{n}=(M_{p_1})^sM_{p_2}k^r=p^sqk^r,$ with $s,r\in\mathbb{N}$ and $p,q,k$ are prime numbers. 
\item [(iii)] if $n=p_1^2$, then $M_{n}=M_{p_1}q^tk^r$ or $M_{n}=p^sq^tk^r$, with $M_{p_1}=p^sq^t$ and $(s,t)=1$, and $p, q, k$ are prime numbers. 
\item [(iv)] if $n=p_1^3$, then $M_{n}=M_{p_1}q^tk^r=pq^tk^r,$ with $t,r\in\mathbb{N}$ and $p,q, k$ are prime numbers. 
\item [(v)] if $n=p_1$, then $M_{n}=p^sq^tk^r$ and $p=2l_1p_1+1, q=2l_2p_1+1,k=2l_3p_1+1$, where $l_1,l_2,l_3$ are distinct positive integers and $l_i\equiv 0\,\,\mbox{or}\,\,-p_1\,\,\pmod 4$, and $\gcd(s,t,r)=1$, with $s, t, r\in\mathbb{N}$.
\end{enumerate}
\end{theorem}
\begin{proof}
This first part is an immediate consequence of the Proposition \ref{P2}.

$(i)$ If $\omega(M_{n})=3$, with $n=2p_1$, then on one hand $M_{n}=p^sq^tk^r$, with $t,s,r\in\mathbb{N}$. On the other hand, according to Proposition \ref{P1}, $\omega(M_{2p_1})>\omega(M_{p_1})+\omega(M_2)$, i.e., $M_{p_1}=q$, according to Theorem \ref{L3}.  We noted that $M_{2p_1}=(2^{p_1}-1)(2^{p_1}+1)$ and $q$ does not divide $2^{p_1}+1$, because if $q|(2^{p_1}+1)$, then $q|2^{p_1}+1-(2^{p_1}-1)=2$. This is a contradiction, since $q$ is an odd prime. Thus, $M_{n}=(M_{2})^sM_{p_1}w^r=3^sqk^r$. Moreover, according to Lemma \ref{L4}, we have $s=1$ if $p_1\neq2^2-1=3$. Therefore, $M_{n}=M_{2}M_{p_1}w^r=3qk^r$.

 $(ii)$ If $\omega(M_{n})=3$, with $n=p_1p_2$,  then on one hand $M_{n}=p^sq^tk^r$, with $t,s,r\in\mathbb{N}$. On the other hand, according to Proposition \ref{P1}, $\omega(M_{p_1p_2})>\omega(M_{p_1})+\omega(M_{p_2})$, i.e., $M_{p_1}=p$ and $M_{p_2}=q$, according to Theorem \ref{L3}. Thus, $M_{n}=(M_{p_1})^s(M_{p_2})^tk^r=p^sq^tk^r$ and $\gcd(s,t,r)=1$ if $s,t,r>1$, according to Theorem \ref{L3}. However, $2^{p_2}-1\nmid p_1$, because $p_1<p_2.$ According to Lemma \ref{L4}, we have $t=1$. Thus, $M_{n}=M_{p_1}^sM_{p_2}k^r=p^sqk^r$.
 
  $(iii)$ If $\omega(M_{n})=3$, with $n=p_1^2$, then on one hand $M_{n}=p^sq^tw^r$, with $t,s,r\in\mathbb{N}$. On the other hand, according to Lemma \ref{L4}, we have $M_{p_1}=p^sq^t$, with $(s,t)=1$ or  $M_{p_1}=p$.  

$(iv)$ If $\omega(M_{n})=3$, with $n=p_1^3$, then on one hand $M_{n}=p^sq^tw^r$, with $t,s,r\in\mathbb{N}$. On the other hand, according to Theorem \ref{L1}, $\omega(M_{p_1^3})>\omega(M_{p_1^2})>\omega(M_{p_1})\geq1$, i.e., $M_{p_1}=p^s$. According to Theorem \ref{L3}, we have $s=1.$ Thus, $M_{n}=M_{p_1}q^tk^r=pq^tk^r.$

$(v)$ If $n=p_1$, then $M_{n}=p^sq^tk^r$, with $t,s,r\in\mathbb{N}$. However, according to Theorem \ref{L3}, $\gcd(s,t,r)=1$. The form of $p,q$ and $k$ is given by Remark \ref{L5}.
\end{proof}

\medskip
We present some examples of solutions for Theorems \ref{T1}, \ref{T2} and \ref{T3}.

\begin{enumerate}
\item[(i)] $\omega(M_n)=1$, where $n$ is a prime number: $M_2=3, M_3=7, M_5=31, M_7=127,\dots$

\item[(ii)] $\omega(M_n)=2$, where $n$ is a prime number: $M_{11}=2047=23\times89, M_{23}=8388607=47\times178481, \dots$; with $n=p^2$, where $p$ is a prime number: $M_4=15=M_2\times5, M_9=511=M_3\times73, M_{49}=M_{7}\times4432676798593,\dots$. 

\item[(iii)] $\omega(M_n)=3$, where $n$ is a prime number: $M_{29}=536870911=233\times1103\times2089, M_{43}=8796093022207=431\times9719\times2099863, \dots$; with $n=2p$, where $p$ is a prime number: $M_{10}=M_2\times M_5\times11, M_{14}=M_2\times M_7\times43,\dots$; with $n=p^3$, $p$ is a prime number: $M_8=255=M_2\times5\times17, M_{27}=M_3\times73\times262657, \dots$; with $n=p_1p_2$, where $p_1$ and $p_2$ are distinct odd prime numbers: $M_{15}=M_3\times M_5\times151, M_{21}=(M_3)^2\times M_7\times337,\dots$; with $n=p^2$, where $p$ is a prime number: $M_{25}=M_5\times601\times1801,\dots$.   
\end{enumerate}


\section{Study of the Equation $M_m+M_n=2p^a$}

We consider the equation $M_m+M_n=2p^a$ with $m,n\geq2$, $p$ an odd prime number and $a$ a positive integer. We present two results on the solutions to this equation. 

\begin{lemma}\label{l1}
For every $p\equiv 1\pmod 4,$ we have $p^a +1=2k,$ where $k$ is an odd integer.
\end{lemma}
\begin{proof}
We have
\begin{eqnarray*}
p\equiv 1\pmod 4 &\Rightarrow &p^a\equiv 1\pmod 4 \\
&\Rightarrow & p^a+1\equiv 2\pmod 4 \\
&\Rightarrow & p^a+1 = 4a+2,\ \mbox{ for some }a\in \mathbb{Z}\\
&\Rightarrow & p^a+1 = 2k;\ \gcd(k,2)=1.
\end{eqnarray*}
\end{proof}

\begin{theorem}\label{t1}
Let $p$ be a prime number with $p\equiv 1\pmod 4.$ Then
\begin{eqnarray}\label{e1}
M_m+M_n=2p^a, \mbox{ with }m\le n,
\end{eqnarray}
has an integer solution only if $p=2^{2^b}+1$. Moreover, given such a prime, the solutions are given by $(m,n,a) = (2,2^b+1,1).$
\end{theorem}

\begin{proof}
Suppose that \eqref{e1} has a solution, say $(m,n,a)$, with $m \le n$.  Notice that
\[2^m+2^n=2p^a+2=4k,\]
by Lemma \ref{l1}.

Since $k$ is odd, it follows that $m=2$ or $n=2.$

\noindent\textbf{\textit{Case 1.}} If $n=2$ then $m\in\{1,2\}$.
Hence we get $4+2=4k$ or $4+4=4k$. Since 6 is not a multiple of 4, we are left with the later case, which implies $k=2$. Therefore, $2k=4=p^a+1$, which is absurd, since $p^a+1\ge 5$.

\noindent\textbf{\textit{Case 2.}} $m=2$ and $n>2$.
From the definition of Mersenne numbers it follows that
\begin{eqnarray*}
4(1+2^{n-2})&=&2p^a+2=4k\\
2(1+2^{n-2})&=& p^a+1=2k\\
p^a+1&=& 2^{n-1}+2.
\end{eqnarray*}
\noindent\textbf{\textit{Case 2.1.}} $a=1.$
\bigskip
So we have, $p=2^{n-1}+1$. We know that if $2^N+1$ is a prime number then $N$ is a power of 2. Hence there exists $b$ such that $n-1=2^b,$ i. e., $2+2^{2^b}=2k$. Hence, if $ (m,n,1)$ is a solution of (\ref{e1}) then $m=2$ and $n=2^b+1$ such that $2^{2^b}+1$ is a prime number. 

\noindent\textbf{\textit{Case 2.2.}} $a\geq 2$.
Suppose that there exists $a\geq2$ satisfying the equation \eqref{e1}. This implies that $p^a=2^{n-1}+1.$ Let us study when $a$ is even and odd separately:

\noindent\textbf{\textit{Case 2.2.1.}} $a$ is even. The equation $p^a=2^{n-1}+1$ implies

\begin{equation}\label{eq2}
\left(p^{\frac{a}{2}}-1\right)\left(p^{\frac{a}{2}}+1\right)=2^{n-1}.
\end{equation}

Let $x$ and $y$ be positive integers such that $p^{\frac{a}{2}}-1=2^x$ and $p^{\frac{p}{2}}+1=2^y$, then $y>x$ and $x+y=n-1$. Thus $2^y-2^x=2^x(2^{y-x}-1)=2$ which implies $x=1, y=2,$  and consequently $n=4$. Therefore $p^a=9$, i.e., $p=3$ and $a=2$, which is absurd, because $3 \not\equiv 1 \pmod 4.$

\noindent\textbf{\textit{Case 2.2.2.}} $a$ is odd. There exists a natural number $l$ such that $a=2l+1$. Thus 

$$p^a=2^{n-1}+1 \Rightarrow p^a-1=2^{n-1} \Rightarrow (p-1)\left(1+p+p^2+\cdots p^{2l-1}+p^{2l}\right)=2^{n-1}.$$

Thus, there exist positive integers $x$ and $y$ such that $p-1=2^x$ and $1+p+p^2+\cdots p^{2l-1}+p^{2l}=2^y$. Clearly $y>x$ and $x+y=n-1$. Notice that $2^y-2^x=2+p^2+\cdots+p^{2l-1}+p^{2l}=k,$ where $k$ is odd, since $p\equiv 1 \pmod 4.$ This only occurs when $x=0$, that is a contradiction. 
\end{proof}

\begin{Observation}
Since we know that Fermat primes are very rare, from Theorem \ref{t1} we can conclude that solutions are also very rare.
\end{Observation}

Theorem \ref{t1} explores the solution in the case $p\equiv 1 \pmod 4$. The next theorem will explore the solutions in the case $p\equiv 3 \pmod 4$.
\begin{theorem}\label{t2} 
Let $p$ be a prime number with $p\equiv 3\pmod 4.$ Then
\begin{eqnarray}\label{e2}
M_m+M_n=2p^a, \mbox{ with }m\le n,
\end{eqnarray}
has an integer solution only if $p^a=2^{k}(1+2^{n-(k+1)})-1$ with $2^k\mid\mid (p^a+1).$ More precisely, such solutions are given by 
\begin{enumerate}
\item[(i)] $(m,n,a)=(2,4,2)$;
\item[(ii)] $(m,n,a)=(k,k,1)$ if $2^k=p^a+1.$ In that case $M_k=p$ is a Mersenne prime;
\item[(iii)] $(m,n,a)=(k+1,n,a)$ if $p^a+1>2^k.$ 
\end{enumerate}
\end{theorem}
\begin{proof}
Since $p\equiv 3\pmod 4$, there exists $k\in \mathbb{N},k\ge 2$ such that $2^k\mid\mid (p^a+1)$. Note that, if $a$ is even, then $k=1$. If $a$ is odd, then $k\ge 2$, since $4\mid (p^a+1)$. Suppose $(m,n,a)$ is a solution of \eqref{e2}. Without loss of generality we can assume $m\le n$

\textbf{\textit{Case 1.}} $a$ is even.

As mentioned earlier, we can write $p^a+1=2b,\mbox{ where }b$ is an odd integer. Since $p\ge 3$, we have $b\ge 5$. Observe that,
\begin{eqnarray*}
M_m+M_n &=& 2p^a\\
2^m + 2^n &=& 2(p^a+1)\\
2^m + 2^n &=& 2^2b.
\end{eqnarray*}

Hence, $m=2$, which together with the fact that $b\ge 5$, implies $n\ge 3$. Therefore, we have
\begin{eqnarray*}
4(1+2^{n-2}) &=& 2^2b\\
1+2^{n-2} &=& b.
\end{eqnarray*}
Therefore, we can conclude that, $b=1+2^{n-2}$, which in turn implies $p^a+1=2(1+2^{n-2})$ iff $(m,n,a)=(2,n,a)$ is the only solution of the equation \eqref{e2}. But, $p^a+1=2(1+2^{n-2})$, then $p^a - 2^{n-1} = 1$. According to Theorem \ref{L3}, the only solution, with $a$ an even number is $n=4$ and $a=2.$ Hence $p=3.$

\textbf{\textit{Case 2.}} $a$ is odd.
As mentioned earlier, we can write $p^a+1=2^kb;\mbox{ where }b$ is an odd integer and $k\ge 2$.
Observe that,
\begin{eqnarray*}
M_m+M_n &=& 2p^a\\
2^m + 2^n &=& 2(p^a+1)\\
2^m + 2^n &=& 2^{k+1}b.
\end{eqnarray*}

Note that, $b=1$ iff $m=n=k$. Therefore, $p^a+1=2^k$ iff 
$(m,n,a)=(k,k,a)$ is the only solution. But, if $p^a+1=2^k,$ then $2^k - p^a = 1.$ By Theorem \ref{L3}, $a=1.$ Thus the only solution is $(m,n,a) = (k,k,1),$ where $M_k=p$ is a Mersenne prime.

From here on let us assume $b\ge 3$. Since 
$2^m + 2^n = 2^{k+1}b,b\ge 3$ we get $m= k+1$. Since $b$ is odd $n\ge k+2$. Therefore,
\begin{eqnarray*}
2^{k+1}(1+2^{n-(k+1)}) &=& 2^{k+1}b\\
1+2^{n-(k+1)} &=& b
\end{eqnarray*}

Therefore, we conclude that, $b=1+2^{n-(k+1)}$, which in turn implies $p^a+1=2^{k}(1+2^{n-(k+1)})$ iff $(m,n,a)=(k+1,n,a)$.
\end{proof}

\section*{Acknowledgement}
We would like to thank the referees for all their suggestions, which improved the presentation of the paper.

\end{document}